\documentclass{birkjour}
\usepackage[noadjust]{cite}
\usepackage{amssymb}
\usepackage{enumerate}
\usepackage{verbatim}
\usepackage{bbm}
\usepackage{tikz-cd}
\usepackage{color}

\usepackage{graphicx} % Allows including images
\usepackage{booktabs} % Allows the use of \toprule, \midrule and \bottomrule in tables
\usepackage{color}
\usepackage{latexsym,amsmath,amscd,amssymb,graphics}
\usepackage{amsmath,amssymb,mathrsfs,framed,esint,slashed,color}
\usepackage{amsthm}
\usepackage{enumerate}
\usepackage{framed}
\usepackage[all]{xy}
\usepackage{xy}
\usepackage{stackrel}
\usepackage{graphicx}
\usepackage{url}

\usepackage{subcaption}
\usepackage{multirow}
\usepackage{pgfplots, pgfplotstable, booktabs, colortbl, array,multirow}

\usepackage{manfnt}
\usepackage[bbgreekl]{mathbbol}
\usepackage{tikz}

\newtheorem{thm}{Theorem}% [section]

\newtheorem{prop}[thm]{Proposition}
\theoremstyle{definition}

\theoremstyle{remark}
\newtheorem{rem}[thm]{Remark}

\renewcommand{\sc}[2]{\langle #1|#2 \rangle}

\renewcommand{\H}{\mathcal{H}}

\DeclareMathOperator{\Tr}{Tr}

\renewcommand{\L}{\mathcal{L}(\H)}

\renewcommand{\to}{\rightarrow}

\newcommand{\be}{\begin{equation}}
\newcommand{\ee}{\end{equation}}
\newcommand{\bse}{\begin{subequations}}
\newcommand{\ese}{\end{subequations}}
\newcommand{\ben}{\begin{enumerate}}
\newcommand{\een}{\end{enumerate}}
\newcommand{\bit}{\begin{itemize}}
\newcommand{\eit}{\end{itemize}}

\usepackage{skull}

\begin{document}
\bibliographystyle{grzyby}

\title{The restricted Siegel disc as coadjoint orbit}

\author[F. Gay-Balmaz]{Fran\c{c}ois Gay-Balmaz}
\address{\\
\begin{minipage}{0.5\textwidth}
School of Physical and  Mathematical \\
$\phantom{q}$ Sciences \\
Nanyang Technological University \\
21 Nanyang Link \\
Singapore 637371\\
\end{minipage}
\begin{minipage}{0.5\textwidth}
UMR CNRS 8539\\
Laboratoire de M\'et\'eorologie Dynamique \\
Ecole Normale Sup\'erieure  \\
24 rue Lhomond  \\
75005 Paris \\
France
\end{minipage}
}
\email{francois.gb@ntu.edu.sg, francois.gay-balmaz@lmd.ens.fr }

\thanks{FGB: Partially
supported by a Nanyang Technological University start-up grant}

\author[T.S. Ratiu]{Tudor S. Ratiu,} 
\address{\\
\begin{minipage}{0.5\textwidth}
School of Mathematical Sciences \\
Ministry of Education Laboratory \\ 
$\phantom{q}$  of Scientific Computing (MOE-LSC) \\
Shanghai Frontier Science Center \\ 
$\phantom{q}$ of Modern Analysis \\ 
Shanghai Jiao Tong University \\
800 Dongchuan Road \\ 
Minhang District, Shanghai \\
200240 China \\
\end{minipage}
\begin{minipage}{0.5\textwidth}
Section de Math\'ematiques \\
Ecole Polytechnique F\'ed\'erale \\
$\phantom{q}$ de Lausanne \\
CH-1015 Lausanne \\ Switzerland \\
\end{minipage}
}
\email{ratiu@sjtu.edu.cn, tudor.ratiu@epfl.ch}

\thanks{\noindent TSR: Partially
supported by the National Natural Science Foundation of China  grant
number 11871334 and by the Swiss National Science Foundation grant 
NCCR SwissMAP}

\author[A.B.~Tumpach]{Alice B. Tumpach}

\address{\\
\begin{minipage}{0.5\textwidth}
UMR CNRS 8524\\
UFR de Math\'ematiques\\
Laboratoire Paul Painlev\'e\\
59 655 Villeneuve d'Ascq Cedex\\ France \\
\end{minipage}
\begin{minipage}{0.5\textwidth}
Institut CNRS Pauli\\ UMI  CNRS 2842\\ 
Oskar-Morgenstern-Platz 1 \\1090 Wien\\Austria \\
\end{minipage}
}
\email{alice-barbora.tumpach@univ-lille.fr}

\thanks{\noindent ABT: Partially supported by  Fonds zur 
F\"orderung der wissenschaftlichen Forschung, Austria (number I 5015-N) 
grant ``Banach Poisson--Lie groups and integrable systems''.}

\dedicatory{Dedicated to the memory of Anatol Odzijewicz our friend\\
  and founder of this conference series}

\begin{abstract}
  The restricted Siegel disc is a homogeneous space related to the
  connected component $T_0(1)$ of the Universal Teichm\"uller space
  via the period mapping. In this paper we show that it is a coadjoint
  orbit of the universal central extension of the restricted
  symplectic group or, equivalently, an affine coadjoint orbit of the
  restricted symplectic group with cocycle given by the Schwinger
  term.
\end{abstract}
\subjclass{53D17,70H06,35F21,47B99}
\keywords{symmetric space, Siegel disc, restricted Grassmannian, 
universal Teichm\"uller space}

\maketitle

\section{Introduction}

The restricted Siegel disc contains the image of the connected 
component $T_0(1)$ of the Universal Teichm\"uller space by the 
period mapping (see \cite{TaTe2004} or~\cite{Tum17} for an overview).
We show that the restricted Siegel disc is symplectomorphic to a 
one-parameter family of coadjoint orbits of the universal central 
extension of the restricted symplectic linear group. This central 
extension is constructed using the universal central extension 
$\widetilde{\operatorname{GL}_{\rm res}}\to\operatorname{GL}_{\rm{res}}$ 
of the restricted general linear group  $\operatorname{GL}_{\rm{res}}$ 
(see \cite{PrSe1990, Wu2001, GO10}). The principal 
bundle $\widetilde{\operatorname{GL}_{\rm res}}
\to\operatorname{GL}_{\rm{res}}$ does not have a global smooth 
section, hence, as a smooth manifold, 
$\widetilde{\operatorname{GL}_{\rm{res}}}$ is not diffeomorphic 
to $\operatorname{GL}_{\rm{res}}\times\mathbb{C}^\times$ and the 
group multiplication is not described by a global smooth cocycle 
(see  section~\ref{central}). Equivalently, the restricted Siegel 
disc can be seen as an affine coadjoint orbit of the restricted 
symplectic linear group. 

The paper is organized as follows. In section~\ref{intro} we 
introduce the relevant Banach Lie groups and present some of 
their geometric properties. In section~\ref{Siegel} we define 
the restricted Siegel disc. In section~\ref{central} we present 
the relevant central extensions of Banach Lie groups. Finally, in 
section~\ref{spres}, we identify some coadjoint orbits of the 
central extension of the restricted symplectic group with the 
Siegel disc and compute the corresponding symplectic form.

\section{Some Banach Lie groups and their geometric properties}
\label{intro}
\subsection{The symplectic linear group 
$\operatorname{Sp}(\mathcal{V},\Omega)$}

Consider a real infinite-dimensional separable Hilbert space 
$\mathcal{V}$ with inner product 
$\langle\cdot, \cdot\rangle_{\mathcal{V}}$ 
and denote by $\operatorname{GL}(\mathcal{V})$ the general 
linear group of real linear bounded invertible operators 
on $\mathcal{V}$; $\operatorname{GL}(\mathcal{V})$ 
is a Banach Lie group whose Banach Lie algebra 
$\mathfrak{gl}(\mathcal{V}):= L(\mathcal{V})$ is the Banach space of 
all real linear bounded  operators on $\mathcal{V}$.
Suppose that $\mathcal{V}$ has a strong symplectic form $\Omega$, 
i.e., $\mathcal{V} \ni v \mapsto \Omega(v, \cdot) 
\in \mathcal{V}^\ast$ is a Banach space isomorphism. 
The \textbf{symplectic linear group} is 
\[
\operatorname{Sp}(\mathcal{V},\Omega)
:=\{a\in\operatorname{GL}(\mathcal{V})\mid 
\Omega(au,av)=\Omega(u,v),\;\;
\text{for all}\;\; u,v\in\mathcal{V}\}.
\]
By the Riesz theorem one can define an operator 
$J\in\operatorname{GL}(\mathcal{V})$ by
\begin{equation}\label{def de J /Omega et pdt scalaire}
\Omega(u,v)=\langle u,Jv\rangle_{\mathcal{V}}, \;\;\text{for all}
\;\; u,v\in\mathcal{V}.
\end{equation}
Note that we have $a\in\operatorname{Sp}(\mathcal{V},\Omega)$ 
if and only if $a\in \operatorname{GL}(\mathcal{V})$ and
\[
a^TJa=J.
\]
The Lie algebra of the symplectic group is
\[
\mathfrak{sp}(\mathcal{V},\Omega):=
\{A\in\mathfrak{gl}(\mathcal{V})\mid A^TJ=-JA\}.
\]
We assume, additionally, that the symplectic form $\Omega$ is
a K\"ahler form with respect to 
$\langle\cdot, \cdot\rangle_{\mathcal{V}}$;
in other words, $J$ is a complex structure compatible with 
$\langle\cdot, \cdot\rangle_{\mathcal{V}}$:
\begin{equation}\label{compatibilite de J avec pdt sca}
J^2=-\operatorname{id}_\mathcal{V}\quad \text{and}\quad 
\langle Ju,Jv\rangle_{\mathcal{V}}=\langle u,v\rangle_{\mathcal{V}}.
\end{equation}
So, if $\Omega$ is K\"ahler, we have
\[
a^TJa=J \Longleftrightarrow aJa^T=J \quad \text{and}
\quad A^TJ=-JA \Longleftrightarrow  AJ=-JA^T.
\] 
Let us now consider the complexification $\mathcal{H} := 
\mathcal{V}\otimes \mathbb{C}$ of $\mathcal{V}$ with hermitian 
scalar product
\[
\langle u, v\rangle_\mathcal{H} := 
\langle \Re u, \Re v\rangle_{\mathcal{V}} + 
\langle \Im u, \Im v\rangle_{\mathcal{V}} + 
i \langle \Im u, \Re v\rangle_{\mathcal{V}} - 
i \langle \Re u, \Im v\rangle_{\mathcal{V}};
\]
that is, $\langle \cdot, \cdot\rangle_\mathcal{H}$ is 
$\mathbb{C}$-linear in the first variable, 
$\mathbb{C}$-anti-linear in the second, and
\[
\langle u, v\rangle_\mathcal{H} = 
\overline{\langle v, u\rangle_\mathcal{H}} = 
\overline{\langle \bar{u}, \bar{v}\rangle_\mathcal{H}},
\]
where $\bar{u}:=\Re u -i\Im u$, 
$\bar{v}:=\Re v -i\Im v$.
Extend $J$ and $\Omega$ by $\mathbb{C}$-linearity to $\mathcal{H}$ 
and decompose $\mathcal{H}$ as
\[
\mathcal{H}=\mathcal{H}_+\oplus\mathcal{H}_-,\quad
\mathcal{H}_+:=\operatorname{Eig}_i(J),\quad
\mathcal{H}_-:=\operatorname{Eig}_{-i}(J),
\]
where $\operatorname{Eig}_{\pm i}(J)$ denotes the eigenspace of 
$J$ associated to $\pm i$. With respect to the decomposition 
$\mathcal{H}=\mathcal{H}_+\oplus\mathcal{H}_-$, the operators 
$a\in \mathfrak{gl}(\mathcal{V})$ extended complex linearly 
to $\mathcal{H}$ have the block form
\begin{equation}\label{matrix_notation}
a=\begin{pmatrix} g&h\\ \bar h&\bar g
\end{pmatrix}, 
\end{equation}
where, for $u \in \mathcal{H}$,
\[
\bar{g}(u):=\overline{g(\bar{u})}
 \quad\text{and}\quad \bar{h}(u):=\overline{h(\bar{u}).}
\]
Note that $g \in L(\mathcal{H}_+,\mathcal{H}_+)$, $\bar{g} \in 
L(\mathcal{H}_-, \mathcal{H}_-)$, but $h \in L(\mathcal{H}_-,\mathcal{H}_+)$ 
and $\bar{h} \in L(\mathcal{H}_+,\mathcal{H}_-)$. In particular, for 
$g = \operatorname{id}_{\mathcal{H}_+}$, we have $\bar{g} = 
\operatorname{id}_{\mathcal{H}_-}$.
 
The expression of the complex bilinear form $\Omega$ with respect 
to the  hermitian scalar product 
$\langle\cdot, \cdot\rangle_\mathcal{H}$ reads
\begin{equation}
\label{complex_omega_j_inner_product_relation}
\Omega(u, v) = \langle u, \overline{J v}\rangle_\mathcal{H} 
= \langle J v, \overline{u}\rangle_\mathcal{H}.
\end{equation} 
Note that
$a\in\operatorname{Sp}(\mathcal{V},\Omega)$ if and only if 
$a\in\operatorname{GL}(\mathcal{V})$ and 
\[
\Omega(a u, a v) = \Omega(u, v)
\]
for all $u, v \in \mathcal{H}$. Since for 
$a\in\operatorname{GL}(\mathcal{V})$, one has 
$\overline{a u} = a(\overline{u})$, 
\eqref{complex_omega_j_inner_product_relation} leads to
\[
\Omega(a u, a v) = \langle J av, \overline{a u}\rangle_\mathcal{H} 
= \langle J av, a(\overline{u})\rangle_\mathcal{H} 
= \langle a^* J a v, \overline{u}\rangle_\mathcal{H}
\]
for all $u, v \in \mathcal{H}$. Hence
\begin{equation}\label{Sp_condition}
a \in \operatorname{Sp}(\mathcal{V},\Omega) \Longleftrightarrow 
\begin{cases} 
a \in \operatorname{GL}(\mathcal{V})\subset 
\operatorname{GL}(\mathcal{H})\\ 
a^* J a = J.
\end{cases}
\end{equation}
Using the block decomposition 
\[
J=\begin{pmatrix} i&0\\0& -i
\end{pmatrix}
\]
with respect to the direct sum $\mathcal{H} = \mathcal{H}_+\oplus 
\mathcal{H}_-$, we conclude that  conditions 
\textcolor{red}{ \eqref{Sp_condition} } 
are equivalent to $a = \begin{pmatrix} g&h\\ \bar h&\bar g
\end{pmatrix}$ with
\begin{equation}\label{symplectic_conditions}
g^*g-h^T\bar{h}=\operatorname{id}_{\mathcal{H}_+}  \quad 
\textrm{and} \quad 
g^*h=h^T\bar{g}\in L(\mathcal{H}_-, \mathcal{H}_+),
\end{equation}
where the transpose is defined by
\begin{equation}\label{transpose}
h^T:=(\bar{h})^*  \in L(\mathcal{H}_-, \mathcal{H}_+).
\end{equation}
Thus, for $g = \operatorname{id}_{\mathcal{H}_+}$, we have 
$g^T = (\bar{g})^* = \operatorname{id}_{\mathcal{H}_-}$.
Note that \eqref{Sp_condition} implies
that $a^{-1}=-JaJ$ which in block form has the expression
\begin{equation}\label{rel_symp}
\begin{pmatrix} g&h\\ \bar h&\bar g
\end{pmatrix}^{-1}= \begin{pmatrix} g^*&-h^T\\ -h^*&g^T
\end{pmatrix},
\end{equation}
or, explicitly,
\begin{equation}
gg^*-hh^*=\operatorname{id}_\mathcal{H_+} \text{ and } 
 gh^T=hg^T  \in L(\mathcal{H}_-, \mathcal{H}_+).
\end{equation}
Using \eqref{Sp_condition}, the Lie algebra of the symplectic group satisfies
\begin{equation}\label{sp_condition}
A \in \mathfrak{sp}(\mathcal{V},\Omega) \Longleftrightarrow 
\begin{cases}
A \in \mathfrak{gl}(\mathcal{V})\subset  \mathfrak{gl}(\mathcal{H})\\ 
A^* J + J A = 0.
\end{cases}
\end{equation}
Relative to the decomposition $\mathcal{H}=
\mathcal{H}_+\oplus\mathcal{H}_-$, we have
\[
\mathfrak{sp}(\mathcal{V},\Omega)=\left\{\left.
\begin{pmatrix}
A_1       &A_2\\
\bar{A}_2 &\bar{A}_1
\end{pmatrix}
\in\mathfrak{gl}(\mathcal{H}) \,\right|\,
A_1^*=-A_1,\;A_2=A_2^T\right\}.
\]

\subsection{ The restricted linear group 
${\rm GL}_{\rm res}$}

The \textbf{restricted general linear group}  relative to the 
polarization $\mathcal{H}=\mathcal{H}_+\oplus\mathcal{H}_-$ 
(see \cite{PrSe1990, Wu2001}) is defined by
\[
\operatorname{GL}_{\rm{res}}:=
\left\{a\in\operatorname{GL}(\mathcal{H})\mid 
\left[d,a\right]\in L^2(\mathcal{H})\right\},
\]
where $\operatorname{GL}(\mathcal{H})$ is the 
Banach Lie group of linear complex bounded invertible operators 
on $\mathcal{H}$, $d$ is defined by
$$d:=i(p_+-p_-),$$
$p_\pm$ denotes the projection of $\mathcal{H}$ onto $\mathcal{H}_\pm$ 
parallel to $\mathcal{H}_\mp$, and $L^2(\mathcal{H})$ 
is the space of Hilbert-Schmidt operators on $\mathcal{H}$. Recall
that the Hilbert-Schmidt norm $\| A \|_2$ of a linear operator
$A$ on $ \mathcal{H}$ is defined by
\[
\|A\|_2^2 = \sum_{n=1}^\infty \|Ae_n\|^2,
\]
where $\{e_n\mid n \in \mathbb{N}\}$ is any Hilbert basis of 
$\mathcal{H}$, that $L^2(\mathcal{H}):=
\{A \in \mathfrak{gl}(\mathcal{H}) \mid \|A\|_2<\infty\}$
is a double sided ideal in $\mathfrak{gl}(\mathcal{H})$, 
and that Hilbert-Schmidt operators are compact. 

Note that $g\in\operatorname{GL}(\mathcal{H})$ can be written in the
form
\[
a=
\begin{pmatrix}
a_{++}    &a_{+-}\\
a_{-+} &a_{--}
\end{pmatrix},
\]
where $a_{++}\in L(\mathcal{H}_+,\mathcal{H}_+)$, 
$a_{-+}\in L(\mathcal{H}_+,\mathcal{H}_-)$, 
$a_{+-}\in L(\mathcal{H}_-,\mathcal{H}_+)$, and 
$a_{--}\in L(\mathcal{H}_-,\mathcal{H}_-)$. Using this notation, 
for any $a\in\operatorname{GL}(\mathcal{H})$, we have
\begin{align*}
a\in\operatorname{GL}_{\rm{res}}&\Longleftrightarrow
\left[d,a\right]\in L^2(\mathcal{H})\\
&\Longleftrightarrow a_{+-}\in L^2(\mathcal{H}_-,\mathcal{H}_+)\;\;
\text{and}\;\; a_{-+}\in L^2(\mathcal{H}_+,\mathcal{H}_-).
\end{align*}
Since Hilbert-Schmidt operators are a double sided 
ideal in the Lie algebra of all bounded operators, we have
$ab\in\operatorname{GL}_{\rm{res}}$ for all $a,b\in
\operatorname{GL}_{\rm{res}}$. Moreover, if $a$ is 
invertible and belongs to $\operatorname{GL}_{\rm{res}}$, its inverse 
$b:= a^{-1}$ belongs also to $\operatorname{GL}_{\rm{res}}$. Indeed, 
from $ab = \textrm{id}_{\mathcal{H}}$ it follows that 
\begin{equation}\label{ab}
a_{++} b_{+-} + a_{+-} b_{--} = 0,
\end{equation}
and from $ba = \textrm{id}_{\mathcal{H}}$ we have
\begin{equation}\label{ba}
b_{++} a_{++} = \textrm{id}_{\mathcal{H}_+} -  b_{+-} a_{-+}.
\end{equation}
Multiplying equation \eqref{ab} from the left by $b_{++}$ and using 
\eqref{ba} we get
\[
\left(\textrm{Id}_{\mathcal{H}_+} -  b_{+-} a_{-+}\right) b_{+-} 
+ b_{++} a_{+-} b_{--} = 0,
\]
which can be rewritten as
\[
b_{+-}  = b_{+-} a_{-+} b_{+-} - b_{++} a_{+-} b_{--}.
\]
Since $a_{-+}$ and $a_{+-}$ are Hilbert-Schmidt, it follows from the double
sided ideal property of Hilbert-Schmidt operators that $b_{+-}$ is 
Hilbert-Schmidt. Similarly one can show that $b_{-+}$ is Hilbert-Schmidt. 
Consequently, $\operatorname{GL}_{\rm{res}}$ is a Banach Lie group.

Since $a\in\operatorname{GL}_{\rm res}$ is invertible and 
$a^{-1}\in\operatorname{GL}_{\rm res}$, we conclude from 
equation~\eqref{ba} and its relative that $a_{++}$ and $a_{--}$
are invertible modulo compact operators and are therefore 
Fredholm operators; see \cite[Theorem VII.2]{Pa1963}.
Moreover we have
\[
\operatorname{Ind}a_{++}=-\operatorname{Ind}a_{--},
\]
where $\operatorname{Ind}A$ is the Fredholm index of the 
operator $A$, defined by
\[
\operatorname{Ind}A:=\operatorname{dim}\operatorname{ker}A -
\operatorname{codim}\operatorname{range}A.
\]

It is known (see \cite[\S 6.2]{PrSe1990}) that
$\operatorname{GL}_{\rm{res}}$ is a
Banach Lie group and an open subset of the Banach algebra
\[
\mathfrak{gl}_{\rm{res}}=\left\{A\in\mathfrak{gl}(\mathcal{H})\mid
\left[d,A\right]\in L^2(\mathcal{H})\right\}
\]
endowed with the norm
\[
\|A\|_{\rm res}:=\|A\|+\left\|\left[d,A\right]\right\|_2.
\]
Unlike $\operatorname{GL}(\mathcal{H})$, the restricted general
linear group
$\operatorname{GL}_{\rm{res}}$ has
infinitely many connected components. More precisely, $a$ and $b$
are in the same connected component if and only if
$\operatorname{Ind}a_{++}=\operatorname{Ind}b_{++}$; 
see \cite[\S 6.2 ]{PrSe1990}.

\subsection{The  restricted symplectic linear group 
${\rm Sp}_{\rm res}$}

The \textbf{restricted
symplectic group} is
\begin{align*}
\operatorname{Sp}_{\rm{res}}(\mathcal{V},\Omega)&:=
\operatorname{Sp}(\mathcal{V},\Omega)\cap\operatorname{GL}_{\rm{res}}\\
&\;=\left\{\left.
\begin{pmatrix}
g            &h\\
\bar{h} &\bar{g}\end{pmatrix}
\in\operatorname{Sp}(\mathcal{V},\Omega) \,\right|\, 
h\in L^2(\mathcal{H}_-,\mathcal{H}_+)\right\},
\end{align*}
and its Lie algebra is given by
\[
\mathfrak{sp}_{\rm{res}}(\mathcal{V},\Omega):=
\left\{\left.
\begin{pmatrix}
A_1            &A_2\\
\bar{A}_2 &\bar{A}_1\end{pmatrix}
\in\mathfrak{sp}(\mathcal{V},\Omega)\,\right|\, 
A_2\in L^2(\mathcal{H}_-,\mathcal{H}_+)\right\}.
\]
Using the identities $g^*g=\operatorname{id}_\mathcal{H}+h^T\bar{h}$ 
and $gg^*=\operatorname{id}_\mathcal{H}+hh^*$ we conclude that 
$g^*g$ and $gg^*$ are positive definite and, therefore, injective. 
This proves that $g$ and $g^*$ are injective. Since $g$ is Fredholm, 
its range is closed and equals $(\operatorname{ker}g^*)^\perp=
\mathcal{H}_+$. This proves that $g$ is a bijection and has 
Fredholm index zero. Thus, 
$\operatorname{Sp}_{\rm{res}}(\mathcal{V},\Omega)$ is contained 
in the connected component of the identity of 
$\operatorname{GL}_{\rm{res}}$.

\section{The Siegel disc and restricted Siegel disc}\label{Siegel}

The \textbf{Siegel disc} is defined by
\[
\mathfrak{D}(\mathcal{H}):=\{Z\in L(\mathcal{H}_-,\mathcal{H}_+)
\mid  Z^T = Z\; \text{and}\;
\operatorname{id}_{\mathcal{H}_+}-Z\bar Z\succ 0\},
\]
where, for $A\in L(\mathcal{H}_+,\mathcal{H}_+)$ the notation
$A\succ 0$ means $\langle Au,u\rangle_\mathcal{H}>0$, for all
$u\in\mathcal{H}_+, u\neq 0$ (recall that
$\bar{Z} \in L(\mathcal{H}_+,\mathcal{H}_-)$ and hence
$Z^T=(\bar{Z})^* \in L(\mathcal{H}_-,\mathcal{H}_+)$).
Using the identity
\begin{equation}\label{ZZ}
\left(\operatorname{id}_{\mathcal{H}_+} - Z\bar Z\right)^T 
= \operatorname{id}_{\mathcal{H}_-} - Z^*Z
\end{equation}
we have, equivalently,
\[
\mathfrak{D}(\mathcal{H}):=\{Z\in L(\mathcal{H}_-,\mathcal{H}_+)
\mid  Z^T = Z\; \text{and}\;
\operatorname{id}_{\mathcal{H}_-}-Z^* Z\succ 0\},
\]
The \textbf{restricted Siegel disc} is, by definition,
\[
\mathfrak{D}_{\rm{res}}(\mathcal{H}):=
\{Z\in \mathfrak{D}(\mathcal{H})\mid 
Z\in L^2(\mathcal{H}_-,\mathcal{H}_+)\},
\]
and is an open set in the complex Hilbert space
\begin{equation}\label{def_E_star}
\mathcal{E}_{\rm{res}}(\mathcal{H}):=
\{V\in L^2(\mathcal{H}_-,\mathcal{H}_+)\mid V^T = V\}.
\end{equation}
Since $Z\in\mathfrak{D}_{\rm{res}}(\mathcal{H})$ is a
Hilbert-Schmidt operator, it follows that 
$\operatorname{id}_{\mathcal{H}_+}-Z\bar Z$ is a Fredholm
operator with index zero. Since $\operatorname{id}_{\mathcal{H}_+}-Z\bar Z$ 
is also positive
definite, it is a bijection. We have the following Theorem. 
\begin{thm}
\label{disc restreint de Siegel = espace homogene} 
The map
\begin{equation}\label{Sp_action}
\begin{aligned}
\rho:& \operatorname{Sp}_{\rm res}(\mathcal{V},\Omega)
\times\mathfrak{D}_{\rm res}(\mathcal{H})\longrightarrow
\mathfrak{D}_{\rm res}(\mathcal{H}),\\ 
 &\left(
\begin{pmatrix}
g            &h\\
\bar{h} &\bar{g}
\end{pmatrix}
,Z\right)\longmapsto (gZ+h)(\bar{h}Z+\bar{g})^{-1}
\end{aligned}
\end{equation}
is a well-defined transitive action of the restricted symplectic 
group onto the restricted Siegel disc. The isotropy group of $0$ 
is isomorphic to the unitary group
\[
\operatorname{Sp}_{\rm
res}(\mathcal{V},\Omega)_0\simeq\operatorname{U}(\mathcal{H}_+).
\]
Therefore, the orbit map induces a diffeomorphism of real Banach manifolds
\begin{equation}\label{diffeo1}
\operatorname{U}(\mathcal{H}_+)\backslash
\operatorname{Sp}_{\rm res}(\mathcal{V},\Omega)
\longrightarrow \mathfrak{D}_{\rm res}(\mathcal{H}),\quad 
\left[
\begin{pmatrix}
g       &h\\
\bar{h} &\bar{g}\end{pmatrix}
\right]_{\operatorname{U}(\mathcal{H}_+)}\longmapsto h\bar{g}^{-1}.
\end{equation}
\end{thm}

\begin{proof} In order to show that  in \eqref{Sp_action} 
the operator $(\bar{h}Z+\bar{g})$ is invertible,  we note that, by 
a direct computation,
\begin{align}\label{positive_definite}
&(\bar{h}Z+\bar{g})^*(\bar{h}Z+\bar{g})-(gZ+h)^*(gZ+h) \nonumber \\
&\quad =- Z^*(g^*g - h^T\bar{h})Z + Z^*(h^T\bar{g} - g^*h)
+(g^T\bar{h} - h^*g)Z + g^T\bar{g} - h^*h \nonumber \\
&\quad \stackrel{\eqref{symplectic_conditions}}=
\operatorname{id}_{\mathcal{H}_-}-Z^*Z\succ 0.
\end{align}
It follows that 
\begin{align}\label{positive_definite2}
&(\bar{h}Z+\bar{g})^*(\bar{h}Z+\bar{g}) = 
\operatorname{id}_{\mathcal{H}_-}-Z^*Z + (gZ+h)^*(gZ+h)
\end{align}
is strictly positive, as a sum of a strictly positive operator 
and a positive operator, and is hence injective.
Therefore, we conclude that  $(\bar{h}Z+\bar{g})$ is also injective. 
Since $\bar{g}$ is a Fredholm operator with Fredholm index zero 
and $\bar{h}Z$ is compact, $(\bar{h}Z+\bar{g})$ is also a Fredholm 
operator with Fredholm index zero. It follows  that it is 
a Hilbert space isomorphism of $\mathcal{H}_-$. Since 
the Hilbert-Schmidt operators are a two-sided ideal, we have
\[
(gZ+h)(\bar{h}Z+\bar{g})^{-1}\in L^2(\mathcal{H}_-,\mathcal{H}_+).
\]
Using  the implication 
$A\succ 0\Rightarrow B^*AB\succ 0$ for $B$ injective, 
we obtain, with $A$ the left hand side of \eqref{positive_definite}  
and $B = (\bar{h}Z+\bar{g})^{-1}$, 
\[
\operatorname{id}_{\mathcal{H}_-}-\left[(gZ+h)(\bar{h}Z+\bar{g})^{-1}\right]^*(gZ+h)(\bar{h}Z+\bar{g})^{-1}\succ 0.
\]
Moreover, a direct computation shows that 
$$[(gZ+h)(\bar{h}Z+\bar{g})^{-1}]^T
= (gZ+h)(\bar{h}Z+\bar{g})^{-1}.$$
This proves that $(gZ+h)(\bar{h}Z+\bar{g})^{-1}\in
\mathfrak{D}_{\rm res}(\mathcal{H})$.

The isotropy group of $Z=0$ is easily computed to be 
\[
\operatorname{Sp}_{\rm res}(\mathcal{V},\Omega)_0=\left\{\left.
\begin{pmatrix}
g &0\\
0 &\bar{g}\end{pmatrix}
\right| g\in\operatorname{U}(\mathcal{H}_+)\right\}
\simeq \operatorname{U}(\mathcal{H}_+),
\]
which is a Banach Lie subgroup of  
$\operatorname{Sp}_{\rm res}(\mathcal{V},\Omega)$ since 
its Lie algebra
\[
\mathfrak{sp}_{\rm res}(\mathcal{V},\Omega)_0=\left\{\left.
\begin{pmatrix}
A &0\\
0 &\bar{A}\end{pmatrix}
\right| A\in\mathfrak{u}(\mathcal{H}_+)\right\}
\simeq \mathfrak{u}(\mathcal{H}_+)
\]
admits a topological complement $\mathfrak{m}$ in 
$\mathfrak{sp}_{\rm res}(\mathcal{V},\Omega)$ given by
\[
\mathfrak{m}=\left\{\left.
\begin{pmatrix}
0       &A\\
\bar{A} &0\end{pmatrix}
\right| A\in L^2(\mathcal{H}_-,\mathcal{H}_+)\right\}.
\]
By \cite[Proposition~11, \S 1.6, Chap.III]{Bo1972}, the quotient
space $\operatorname{U}(\mathcal{H}_+)\backslash
\operatorname{Sp}_{\rm res}(\mathcal{V},\Omega)$ is naturally 
endowed with a smooth Banach manifold structure.

For $Z\in L^2(\mathcal{H}_-,\mathcal{H}_+)$ such that 
$\operatorname{id}_{\mathcal{H}_-}-Z^*Z\succ 0$, define
\begin{equation}\label{action_transitive}
B_Z:=Z\frac{\operatorname{argtanh}|Z|}{|Z|}=
Z\sum_{k=0}^\infty\frac{1}{2k+1}(Z^*Z)^k\in 
L^2(\mathcal{H}_-,\mathcal{H}_+).
\end{equation}
When $Z^T=Z$ we have $B_Z^T=B_Z$. This proves that
\[
\begin{pmatrix}
0 &B_Z\\
\overline{B_Z} &0\end{pmatrix}
\in\mathfrak{sp}_{\rm res}(\mathcal{V},\Omega)\quad
\text{and}\quad g_Z:=\operatorname{exp}
\begin{pmatrix}
0 &B_Z\\
\overline{B_Z} &0\end{pmatrix}
\in\operatorname{Sp}_{\rm res}(\mathcal{V},\Omega).
\]
Since
\[
\rho _{g_{Z}}(0)=Z,
\]
where $\rho $ is the action \eqref{Sp_action} (see \cite{Up85} 
and \cite[Remark 6.5]{Ne2002}), we conclude that 
$\operatorname{Sp}_{\rm res}(\mathcal{V},\Omega)$ acts 
transitively on $\mathfrak{D}_{\rm res}(\mathcal{H})$. 
Therefore, the orbit map
\[
\begin{pmatrix}
g            &h\\
\bar{h} &\bar{g}\end{pmatrix}
\in  \operatorname{Sp}_{\rm res}( \mathcal{V} , \Omega ) 
\mapsto h\bar g ^{-1} \in  \mathfrak{D}   _{\rm res}( \mathcal{H} ),
\]
at $Z=0$, induces a smooth bijection
\[
\operatorname{U}(\mathcal{H}_+)\backslash
\operatorname{Sp}_{\rm res}(\mathcal{V},\Omega)
\longrightarrow \mathfrak{D}_{\rm res}(\mathcal{H}),
\]
which is also a diffeomorphism since the tangent map at 
$[\operatorname{id}_{\mathcal{H}}]$ is the isomorphism given by
\[
\mathfrak{u}(\mathcal{H}_+)\backslash
\mathfrak{sp}_{\rm{res}}(\mathcal{V},\Omega)
\longrightarrow \mathcal{E}_{\rm{res}}(\mathcal{H}),
\]
\begin{equation}\label{tgt_map_D_res}
\left[
\begin{pmatrix}
A_1       &A_2\\
\bar{A}_2 &\bar{A}_1\end{pmatrix}
\right]_{\mathfrak{u}(\mathcal{H}_+)}
=\left[
\begin{pmatrix}
0       &A_2\\
\bar{A}_2 &0\end{pmatrix}
\right]_{\mathfrak{u}(\mathcal{H}_+)}\longmapsto A_2,
\end{equation}
where the bracket denotes the equivalence class of the matrix modulo an element
\[
\begin{pmatrix}
A &0\\
0 &\bar{A}\end{pmatrix}, \quad A\in\mathfrak{u}(\mathcal{H}_+).
\tag*{\qedsymbol}
\]
\renewcommand{\qedsymbol}{} 
\end{proof}

\begin{rem}\label{action_de_Sp_2}{
The proof of Theorem~\ref{disc restreint de Siegel = espace
homogene} shows that  $\mathfrak{D}_{\rm res}(\mathcal{H})$ 
is, in fact, a homogeneous space under the Hilbert Lie group
$$\operatorname{Sp}_{2}(\mathcal{V}, \Omega):= 
\{\operatorname{id}_\mathcal{H} + L^2(\mathcal{H})\}\cap 
\operatorname{Sp}(\mathcal{V},\Omega).$$
Indeed, one sees immediately that $g_Z$
belongs to $\operatorname{Sp}_{2}(\mathcal{V}, \Omega)$ whenever
$Z\in\mathfrak{D}_{\rm res}(\mathcal{H})$. 
\hfill $\lozenge$}
\end{rem}

\begin{rem}\label{Sp_mod_U}{
Note also that $\mathfrak{D}_{\rm res}(\mathcal{H})$, as an open 
subset of $\mathcal{E}_{\rm res}(\mathcal{H})$,  is a complex 
Hilbert manifold on which 
$\operatorname{Sp}_{\rm res}(\mathcal{V},\Omega)$ acts by 
biholomorphic maps. Therefore, the orbit map \eqref{diffeo1} 
induces a $\operatorname{Sp}_{\rm res}(\mathcal{V},\Omega)$-invariant 
complex Hilbert manifold structure on the homogeneous space  
$\operatorname{U}(\mathcal{H}_+)\backslash
\operatorname{Sp}_{\rm res}(\mathcal{V},\Omega)$.
\hfill $\lozenge$}
\end{rem}

Note that on $\mathcal{E}_{\rm{res}}(\mathcal{H})$, there is a 
well defined Hermitian inner product given by
\[
h_{\mathfrak{D}}(0)(U, V) := \operatorname{Tr}(V^*U)=
\operatorname{Tr}(\bar{V}U),
\]
with $U^T=U$, $V^T=V \in\mathcal{E}_{\rm{res}}(\mathcal{H})$.
Since $h_{\mathfrak{D}}(0)$ is invariant under the isotropy subgroup 
$\operatorname{Sp}_{\rm{res}}(\mathcal{V},\Omega)_0$ at $Z=0$, 
it extends to a 
$\operatorname{Sp}_{\rm{res}}(\mathcal{V},\Omega)$-invariant 
Hermitian metric $h_{\mathfrak{D}}$ on 
$\mathfrak{D}_{\rm{res}}(\mathcal{H})$. 
A tangent vector $\tilde{U}$ at 
$Z=h\bar{g}^{-1}\in\mathfrak{D}_{\rm{res}}(\mathcal{H})$ is 
the image of a tangent vector $U$ at 
$0\in \mathfrak{D}_{\rm res}(\mathcal{H})$ by the differential 
of the action of 
$\begin{pmatrix} g&h\\ \bar h&\bar g\end{pmatrix}
\in \operatorname{Sp}_{\rm{res}}(\mathcal{V},\Omega)$.  A 
direct computation shows that
\begin{align*}
\tilde{U} &=  g U \bar{g}^{-1} - h \bar{g}^{-1} \bar{h} U \bar{g}^{-1}\\
& = \left(\operatorname{id}_{\mathcal{H}_+} -Z\bar{Z}\right)g U \bar{g}^{-1} 
= (g^*)^{-1}U \bar{g}^{-1};
\end{align*}
hence $U = g^* \tilde{U} \bar{g}$, where we have used 
\begin{align*}
\operatorname{id}_{\mathcal{H}_+} -Z\bar{Z}
& = \operatorname{id}_{\mathcal{H}_+} -Z^T\bar{Z} \\ 
&=  \operatorname{id}_{\mathcal{H}_+} - (g^*)^{-1} h^T h \bar{g}^{-1} \\ 
& =  \operatorname{id}_{\mathcal{H}_+} - 
(g^*)^{-1} (g^*g - \operatorname{id}_{\mathcal{H}_+}) \bar{g}^{-1}  
= (g^*)^{-1} g^{-1}.
\end{align*}
It follows that the Hermitian inner product 
$h_{\mathfrak{D}}$ is given at 
$Z=h\bar{g}^{-1}\in\mathfrak{D}_{\rm{res}}(\mathcal{H})$ by
\begin{equation}\label{Poincare_metric}
h_{\mathfrak{D}}(Z)(\tilde{U},\tilde{V})=
\operatorname{Tr}\left(g^T\tilde{V}^* g g^* \tilde{U} \bar{g}\right)   
=  \operatorname{Tr}\left(\overline{g\bar{g}^T\tilde{V}}
g\bar{g}^T\tilde{U}\right),
\end{equation}
where $g\bar{g}^T$ depends only on $Z$ since for 
$Z=h\bar{g}^{-1}$ we have
\[
(\operatorname{id}_{\mathcal{H}_-}-Z^*Z)^{-1}=\bar{g}g^T.
\]
This Hermitian metric makes $\mathfrak{D}_{\rm{res}}(\mathcal{H})$ 
into a K\"ahler Hilbert manifold with the 
$\operatorname{Sp}_{\rm{res}}(\mathcal{V},\Omega)$-invariant 
K\"ahler two-form $\omega_\mathfrak{D}$ given at the origin by
\[
\omega_\mathfrak{D}(0)(U,V)=
\operatorname{Im}(\operatorname{Tr}(\bar{V}U))=
\frac{i}{2}\operatorname{Tr}(\bar{U}V-\bar{V}U).
\]
Using the diffeomorphism \eqref{diffeo1}, we obtain a
$\operatorname{Sp}_{\rm{res}}(\mathcal{V},\Omega)$-invariant
Hermitian metric on
$\operatorname{U}(\mathcal{H}_+)\backslash
\operatorname{Sp}_{\rm{res}}(\mathcal{V},\Omega)$
whose value at $[\operatorname{id}_\mathcal{H}]$ is
\[
h_\mathfrak{D}([\operatorname{id}_\mathcal{H}])\left(\left[
\begin{pmatrix}
A_1       &A_2\\
\bar{A}_2 &\bar{A}_1\end{pmatrix} 
\right]_{\mathfrak{u}(\mathcal{H}_+)},\left[
\begin{pmatrix}
B_1       &B_2\\
\bar{B}_2 &\bar{B}_1\end{pmatrix} 
\right]_{\mathfrak{u}(\mathcal{H}_+)}\right)
=\operatorname{Tr}(\bar{B}_2A_2).
\]
The expression of the K\"ahler two-form is
\begin{multline}\label{omega_D}
\omega_\mathfrak{D}([\operatorname{id}_\mathcal{H}])\left(\left[
\begin{pmatrix}
A_1       &A_2\\
\bar{A}_2 &\bar{A}_1\end{pmatrix} 
\right]_{\mathfrak{u}(\mathcal{H}_+)},\left[
\begin{pmatrix}
B_1       &B_2\\
\bar{B}_2 &\bar{B}_1\end{pmatrix}
\right]_{\mathfrak{u}(\mathcal{H}_+)}\right) \\
 \qquad = \frac{i}{2}\operatorname{Tr}(\bar{A}_2B_2-\bar{B}_2A_2).
\end{multline}

\section{Universal central extensions of some Banach Lie groups} 
\label{central}

In this section we recall the construction of the central 
extension of $ \operatorname{Sp}_{\rm res}( \mathcal{V} , \Omega )$ 
by $\mathbb{S}^1$. A first step is to construct the central extension 
of the connected component $\operatorname{GL}_{\rm res}^0 $ of the 
identity of $\operatorname{GL}_{\rm res}$ by $\mathbb{C}^\times$ 
(see \cite{PrSe1990,  Wu2001, GO10}).

Let us introduce some notation.
Recall that for  a bounded linear operator $A$, the square root 
$(A^*A)^{\frac{1}{2}}$ of $A^*A$ is well defined (see 
\cite[Theorem VI.9]{RS1}). The Schatten class $L^{1}(\mathcal{H}_+)$ 
of trace class operators on $\mathcal{H}_+$ is the subspace of 
$L(\mathcal{H}_+)$  consisting of those 
bounded operators $A$ on $\mathcal{H}_+$ for which 
$\|A\|_1 := \Tr \big( (A^*A)^{\frac{1}{2}} \big) <\infty$.
Denote by $\operatorname{GL}_1(\mathcal{H}_+)$ the Banach 
Lie group of invertible operators which differ from the identity by  
trace class operators: 
\[
\operatorname{GL}_1(\mathcal{H}_+):=
\{a\in \operatorname{GL}(\mathcal{H}_+)\mid 
a-\operatorname{id}_{\mathcal{H}_+}\in L^1( \mathcal{H} _+)\},
\] 
where the topology is induced by 
$\operatorname{id}_{\mathcal{H}_+}+L^1( \mathcal{H} _+)$.
Recall from \cite[Chapter 3]{Si1979} that operators in 
$\operatorname{GL}_1(\mathcal{H}_+)$ have well-defined determinants.

\subsection{The central extension of $\operatorname{GL}_{\rm res}^0 $}
Denote by $\operatorname{GL}_{\rm res}^0 $ the connected component 
of $\operatorname{GL}_{\rm res}$ containing the identity.
Given an operator $ r \in \operatorname{GL}_{\rm res}$ with 
$ \operatorname{Ind}(r _{++})=\pm 1$, we have a well-defined 
action of $ \mathbb{Z}  $ on $\operatorname{GL}_{\rm res}^0$ 
by group homomorphisms given by $(n, a) \mapsto r ^n  \,a\, r ^{-n}$. 
We can form the semidirect product $ \mathbb{Z}  \,\ltimes\, 
\operatorname{GL}_{\rm res}^0$ with multiplication law 
$$(n,a)(m,b)=(n+m, a\, r ^n b \,r ^{-n})$$ 
and we obtain the Banach Lie group isomorphism
\[
\mathbb{Z}  \,\ltimes\, \operatorname{GL}_{\rm res}^0 
\rightarrow\operatorname{GL}_{\rm res},\quad 
( n, a) \mapsto a \,r^n.
\]

Let us recall from \cite[\S 6.6]{PrSe1990}, \cite[\S II.3]{Wu2001},
or \cite[section~2.2]{GO10} the definition of the central extension  
$\widetilde{\operatorname{GL}^0_{\rm res}}$. 
Let
\[
\mathcal{F} : = \left\{(a, b)\in 
\operatorname{GL}_{\rm res}^0\times
\operatorname{GL}(\mathcal{H}_+)\mid a_{++}- b\in
L^1(\mathcal{H}_{+})\right\},
\]
where, with respect to the decomposition $\mathcal{H} = 
\mathcal{H}_+ \oplus \mathcal{H}_-$,
\[
 a=
\begin{pmatrix}
a_{++}    &a_{+-}\\
a_{-+} &a_{--}
\end{pmatrix}.
\]
According to \cite[\S 6.6]{PrSe1990} and 
\cite[Lemma II.18 and Corollary II.19]{Wu2001}, 
 $ \mathcal{F} $ has the following properties. 
\begin{itemize}
\item $ \mathcal{F} $ is a subgroup of $\operatorname{GL}_{\rm res}^0 
\times \operatorname{GL}( \mathcal{H} _+)$.
\item $ \mathcal{F} $ is a Banach Lie group when endowed with 
the topology induced by the injective map 
\begin{equation}
\begin{array}{llll}
\iota: & \mathcal{F} & \longrightarrow & \operatorname{GL}_{\rm res}^0 \times L^1(\mathcal{H}_{+})\\
& (a, b) & \longmapsto & (a, a_{++}-b)
\end{array}
\end{equation}
and the range of this injective map is open.
\item We have the exact sequence of Banach Lie groups
\[
\{1\} \longrightarrow \operatorname{GL}_1(\mathcal{H}_+) 
\overset{\alpha}{\longrightarrow} \mathcal{F} 
\overset{\beta}{\longrightarrow} 
\operatorname{GL}_{\rm res}^0\longrightarrow \{1\},
\]
where $\alpha~:\operatorname{GL}_1(\mathcal{H}_+) \rightarrow 
\mathcal{F}$ is the injection $h \mapsto 
(\operatorname{id}_\mathcal{H}, h)$ and $\beta~:\mathcal{F} 
\rightarrow  \operatorname{GL}_{\rm res}^0$ is the 
projection $(a ,b) \mapsto a$. Therefore 
$\operatorname{GL}_{\rm res}^0 = 
\mathcal{F}/\alpha(\operatorname{GL}_1(\mathcal{H}_+))$.
\item $ \mathcal{F} $ is contractible.
\end{itemize}
Define the Banach Lie group 
\begin{equation}\label{def_tilde_GL}
\widetilde{\operatorname{GL}^0_{\rm res}} := 
\mathcal{F}/\alpha\left(\operatorname{SL}_1(\mathcal{H}_+)\right),
\end{equation}
where $\operatorname{SL}_1(\mathcal{H}_+)$ is the normal
Banach Lie  subgroup of determinant-$1$ operators in 
$\operatorname{GL}_1(\mathcal{H}_+)$. We denote by 
$[a,h]_ \alpha $ the elements in 
$\widetilde{\operatorname{GL}^0_{\rm res}}$. Since 
\[
\{1\} \longrightarrow \mathbb{C}^ \times 
\overset{[\alpha]}{\longrightarrow}  
\widetilde{\operatorname{GL}^0_{\rm res}}
\overset{[ \beta ]}{\longrightarrow}  
\operatorname{GL}_{\rm res}^0\longrightarrow \{1\},
\]
where $[\alpha ](z)=[\operatorname{id}_\mathcal{H}, z\,
\operatorname{id}_{\mathcal{H}_+}]_{ \alpha }$ and 
$ [ \beta ]\left(  [a, b]_ \alpha \right) =a$, is an exact 
sequence of Banach Lie groups and 
$[\alpha] \left(\mathbb{C}^\times \right)$ is in the center, 
it follows that $\widetilde{\operatorname{GL}^0_{\rm res}}$ is 
a central extension of $\operatorname{GL}^0_{\rm res}$ by 
$\mathbb{C}^ \times $.

\subsection{The central extension of $\operatorname{GL}_{\rm res}$}
\label{section_central_gl}

We now recall the construction of the central extension 
$ \widetilde{ \operatorname{GL}_{\rm res}}$ of  the whole group 
$\operatorname{GL}_{\rm res}$. Let $r$ be the shift operator
defined by $r(e_j) := e_{j+1}$, 
where $\{e_j\mid j\geq 0\}$ is a Hilbert basis of $\mathcal{H}_+$ 
and $\{e_j\mid j<0\}$ is a Hilbert basis of $\mathcal{H}_-$.
Consider the map $\Phi : \mathcal{F} \rightarrow \mathcal{F}$ 
defined by 
\[
\Phi (a,b):= (rar^{-1} , b _r ),
\] 
where  $ b _r $ is defined by 
\[
b_r:=
\begin{cases}
r_{|\mathcal{H}_+} \circ b \circ \left(r_{|\mathcal{H}_+}\right)^{-1}
&\text{on} \quad r(\mathcal{H}_+)\\ 
\operatorname{id}  &\text{on} \quad 
\left(r_{|\mathcal{H}_+}(\mathcal{H}_+)\right)^{\perp}=\mathbb{C}e_0.
\end{cases}
\]
Note that this defines a smooth map $b \in \operatorname{GL}(\mathcal{H}_+)
\mapsto b_r \in \operatorname{GL}( \mathcal{H}_+)$ and that 
 \[
 (ra r ^{-1} )_{++}-b _r \in L ^1 ( \mathcal{H}_+)
 \] 
which implies that $ \Phi $ maps $\mathcal{F}$ to $\mathcal{F}$. 
Let us show that $\Phi$  is smooth. Consider the following diagram
\[
\begin{xy}
\xymatrix{
(a,b) \in \mathcal{F}\ar[d]^{ \Phi }\ar[rr]^{\hspace{-1.5cm}\iota}& 
&(a, a_{++} - b) \in \operatorname{GL}_{\rm res}^0 \times 
L^1(\mathcal{H}_{+})\ar[d]^{ \Psi }& &\\
(r ar^{-1}, b_r)  \in \mathcal{F}\ar[rr]^{\hspace{-1.5cm}\iota}& 
&  (r a r^{-1}, (r a r^{-1})_{++} - b_r) \in    
\operatorname{GL}_{\rm res}^0 \times L^1(\mathcal{H}_{+})& &\\
}
\end{xy}
\]
We see that
\[
\iota\circ \Phi = \Psi \circ \iota
\]
with 
\begin{equation}
\begin{array}{llll}
\Psi: & \operatorname{GL}_{\rm res}^0 \times L^1(\mathcal{H}_{+})& 
\longrightarrow & 
\operatorname{GL}_{\rm res}^0 \times L^1(\mathcal{H}_{+})\\
& (a, c) & \longmapsto & 
( r a r^{-1}, (r a r^{-1})_{++}- (a_{++} - c)_r).
\end{array}
\end{equation}
Since $\Psi$ is smooth and since the topology of $\mathcal{F}$ is 
given by the induced topology from $\operatorname{GL}_{\rm res}^0 
\times L^1(\mathcal{H}_{+})$ via the injection $\iota$, it follows 
that $\Phi$ is smooth. It is an injective group 
homomorphism which is not surjective. Given an element $[a',b']_\alpha 
\in \widetilde{\operatorname{GL}^0_{\rm res}}$, we can always find 
$ c \in \alpha ( \operatorname{SL}_1( \mathcal{H} _+))$ such that 
$(a',b')c$ is in the image of the map $\Phi$. This shows that 
$\Phi $ induces a well-defined Banach Lie group isomorphism 
$\bar{\Phi}$ of $\widetilde{\operatorname{GL}^0_{\rm res}}$.

Consider the action of $\mathbb{Z}$ on 
$\widetilde{\operatorname{GL}^0_{\rm res}}$ by group 
homomorphisms given by
\[
\left( n, [a,b]_ \alpha \right)\mapsto n \!\cdot\! [a,b]_\alpha := 
\bar{\Phi}^n ([a,b]_\alpha), \quad n \in \mathbb{Z}.
\]
We now form the associated semidirect product $ \mathbb{Z}  \,\ltimes\, 
\widetilde{\operatorname{GL}^0_{\rm res}}$ with multiplication given by
\[
\left( n,[a,b]_\alpha \right) \left( m,[k,l]_\alpha \right) 
=\left( n+m, [a,b]_\alpha (n \!\cdot\! [k,l] _\alpha )\right).
\]
A central extension of $\operatorname{GL}_{\rm res}$ by 
$\mathbb{C}^\times$ is defined by 
\begin{equation}\label{def_GL_ext}
\widetilde{\operatorname{GL}_{\rm res}}:=\mathbb{Z}  \,\ltimes\, 
\widetilde{\operatorname{GL}_{\rm res}^0}.
\end{equation}

As mentioned in the introduction, there is no global smooth section 
of the principal bundle $\widetilde{\operatorname{GL}_{\rm res}}
\to\operatorname{GL}_{\rm{res}}$ (see \cite[page 88]{PrSe1990}). 
Therefore, as a smooth manifold, 
$\widetilde{\operatorname{GL}_{\rm{res}}}$ is not diffeomorphic 
to $\operatorname{GL}_{\rm{res}}\times\mathbb{C}^\times$ and the 
group multiplication is not described by a global smooth cocycle.

The Lie algebra $\widetilde{\mathfrak{gl}_{\rm{res}}}$ of  
$\widetilde{\operatorname{GL}^0_{\rm res}}$ consists of classes 
$[B,q] _\alpha $ in the quotient Banach Lie algebra of 
$\mathcal{F}/\alpha(\operatorname{SL}_1(\mathcal{H}_+))$, 
where $B\in \mathfrak{gl}_{\rm res}$ and 
$q\in \mathfrak{gl}( \mathcal{H} _+)$. Note that 
$\widetilde{\mathfrak{gl}_{\rm{res}}}$ can be identified 
with $\mathfrak{gl}_{\rm res} \oplus \mathbb{C}$ via the map
\[
[B, q]_\alpha  \mapsto \left(B, -2i \operatorname{Tr}(B_{++} - q)\right).
\]
The Lie bracket of 
$\widetilde{\mathfrak{gl}_{\rm{res}}}=
\mathfrak{gl}_{\rm{res}}\oplus\mathbb{C}$
 is given by
\[
[(A,\lambda ),(B,\nu )]=([A,B],s(A,B)),
\]
where $s:\mathfrak{gl}_{\rm{res}}\times\mathfrak{gl}_{\rm{res}}
\to\mathbb{C}$ is the Schwinger term, i.e.,  the continuous 
two-cocycle given by 
\[
s(A,B):=\operatorname{Tr}(A[d,B])= 
2i \operatorname{Tr} \left( A_{-+}B_{+-} - A_{+-}B_{-+}\right) ,
\]
where $d:=i(p_+-p_-)$.
Consider the following Banach Lie algebra
\[
(\mathfrak{gl}_{\rm{res}})_*:=\{\mu\in\mathfrak{gl}_{\rm{res}}
\mid p_\pm\mu|_{\mathcal{H}_\pm}\in L^1(\mathcal{H}_\pm)\},
\]
where $L^1(\mathcal{H}_\pm)$ denotes the Banach space of trace class 
operators on $\mathcal{H}_{\pm}$. 

Consider the restricted trace 
$\operatorname{Tr}_{\rm{res}}: (\mathfrak{gl}_{\rm{res}})_* 
\rightarrow \mathbb{C}$ defined in \cite[formula (2.32)]{GO10} by
\[
\operatorname{Tr}_{\rm{res}} \mu = \operatorname{Tr}(\mu_{++} + \mu_{--}).
\]
By \cite[Proposition~2.1]{GO10}, for $\mu \in (\mathfrak{gl}_{\rm{res}})_*$ 
and $A \in \mathfrak{gl}_{\rm{res}}$, both $\mu A$ and $A \mu$ belong to 
$(\mathfrak{gl}_{\rm{res}})_*$ and $\operatorname{Tr}_{\rm{res}}(\mu A) = 
\operatorname{Tr}_{\rm{res}}(A \mu)$. Consequently 
(see  \cite[formula (2.36)]{GO10})
\begin{equation}
\operatorname{Tr}_{\rm{res}} g \mu g^{-1} = \operatorname{Tr}_{\rm{res}} \mu,
\end{equation}
for any $\mu \in  (\mathfrak{gl}_{\rm{res}})_*$ and any 
$g \in \operatorname{GL}_{\rm{res}}$.
The strongly non-degenerate pairing
\begin{equation}
\!\!\!
\begin{array}{cll}
(\mathfrak{gl}_{\rm{res}})_*\times\mathfrak{gl}_{\rm{res}}& 
\longrightarrow & \mathbb{C},\\
\quad (\mu,A)& \longmapsto& 
\begin{array}{ll} \operatorname{Tr}_{\rm{res}}(\mu A) 
& = \operatorname{Tr}(\mu_{++} A_{++}) + \operatorname{Tr}(\mu_{+-}A_{-+}) \\ 
& + \operatorname{Tr}(\mu_{-+}A_{+-}) + \operatorname{Tr}(\mu_{--} A _{--})\end{array},
\end{array}
\end{equation}
induces an isometric isomorphism of complex Banach spaces
\[
\left((\mathfrak{gl}_{\rm{res}})_*\right)^*
\cong\mathfrak{gl}_{\rm{res}},
\]
i.e., $(\mathfrak{gl}_{\rm{res}})_*$ is a 
predual space of $\mathfrak{gl}_{\rm{res}}$.
Let us define the following duality pairing between $
\left(\mathfrak{gl}_{\rm res}\right)_*\oplus \,\mathbb{C}$ and 
$\mathfrak{gl}_{\rm res} \oplus \,\mathbb{C}$
\begin{equation}\label{pairing}
\begin{array}{lcll}
\langle\cdot, \cdot\rangle : & 
\left((\mathfrak{gl}_{\rm{res}})_*\oplus \, \mathbb{C}\right)
\times\left(\mathfrak{gl}_{\rm{res}}\oplus \, \mathbb{C})\right)& 
\longrightarrow & \mathbb{C},\\
& \quad \left((\mu,\gamma), (B, b)\right)& \longmapsto& 
\operatorname{Tr}_{\rm{res}}(\mu B) + \gamma b.
\end{array}
\end{equation}

\begin{prop}
With respect to the pairing~\eqref{pairing}, 
the coadjoint action 
of $\widetilde{\operatorname{GL}_{\rm{res}}}$ on 
$\left(\widetilde{\mathfrak{gl}_{\rm{res}}}\right)_*
:=\left(\mathfrak{gl}_{\rm{res}}\right)_*\oplus\mathbb{C}$ 
is given by
\begin{equation}\label{action coaff = action co}
\operatorname{Ad}^*_\Gamma(\mu,\gamma)=
( a ^{-1} \mu a-\gamma\sigma(a^{-1}),\gamma),
\end{equation}
where $\Gamma\in\widetilde{\operatorname{GL}_{\rm{res}}}$ projects 
to $a\in\operatorname{GL}_{\rm{res}}$ and $\sigma : 
\operatorname{GL}_{\rm{res}}\to
\left(\mathfrak{gl}_{\rm{res}}\right)_*$ is given by
\[
\sigma(a)=ada^{-1}-d.
\]
The infinitesimal coadjoint action of 
$\widetilde{\mathfrak{gl}_{\rm{res}}}$ is given for 
$(A,\lambda )\in\widetilde{\mathfrak{gl}_{\rm{res}}}$ and 
$(\mu,\gamma)\in\left(\widetilde{\mathfrak{gl}_{\rm{res}}}\right)_*$ 
by
\begin{equation}\label{ad_star}
\operatorname{ad}^*_{(A,\lambda )}(\mu,\gamma)=
\left([\mu,A]-\gamma[d,A],0\right).
\end{equation} 
\end{prop}

\begin{proof} The coadjoint action of the center of 
the extended group is trivial. Therefore, the coadjoint action 
of $\widetilde{\operatorname{GL}_{\rm res}}$ on the predual 
$\left(\widetilde{\mathfrak{gl}_{\rm{res}}}\right)_*$ descends 
to an action of the restricted general group 
$\operatorname{GL}_{\rm res}$. By \eqref{def_GL_ext}, it 
is sufficient to verify equation~\eqref{action coaff = action co}
for $\widetilde{\operatorname{GL}^0 _{\rm res}}$ and for the 
shift $r$. Passing to the quotients, the adjoint action of 
$\widetilde{\operatorname{GL}^0 _{\rm res}}$ on its Lie 
algebra reads
\[
\operatorname{Ad}_\Gamma ([B, q]_\alpha)  = [aBa^{-1}, bq b^{-1}]_\alpha ,
\]
where $\Gamma = [a, b]_\alpha \in 
\widetilde{\operatorname{GL}^0_{\rm res}}$ and 
$[B, q]_\alpha \in \widetilde{\mathfrak{gl}_{\rm res}}$. 
A direct computation shows that 
\[
\operatorname{Tr} \left(\left(aBa^{-1}\right)_{++} - b q b^{-1}\right)
= \operatorname{Tr} (B_{++} - q) + 
\frac{1}{2i}\operatorname{Tr}\left(\sigma(a^{-1})B \right) 
\]
with $\sigma(a) = a d a^{-1} - d$. Consequently 
$[aBa^{-1}, bq b^{-1}]_\alpha$ is identified with
\[
\left(aBa^{-1} , -2i \operatorname{Tr}(B_{++} - q) - 
\operatorname{Tr}\left( \sigma(a^{-1}) B\right) \right) 
\in \mathfrak{gl}_{\rm res}(\mathcal{H}) \oplus \mathbb{C},
\]
and $\Gamma = [a, b]_\alpha \in
\widetilde{\operatorname{GL}_{\rm res}^0}$ acts on 
$(B, \nu)\in \mathfrak{gl}_{\rm res}(\mathcal{H}) \oplus 
\mathbb{C}$ by
\begin{equation}\label{action_adjointe_ds_extension}
\operatorname{Ad}_\Gamma(B, \nu) = \left( aBa^{-1}, \nu - 
\operatorname{Tr}\left( \sigma(a^{-1})B \right) \right).
\end{equation}
Let us note that \eqref{action_adjointe_ds_extension} 
coincides with \cite[formula (2.27)]{GO10}.

Using the duality pairing \eqref{pairing} between 
$\left(\mathfrak{gl}_{\rm res}\right)_*\oplus \,
\mathbb{C}\ni ( \mu , \gamma )$ and 
$\mathfrak{gl}_{\rm res} \oplus \,\mathbb{C}\ni (B,b)$, 
one has
\begin{align*}
\left\langle\operatorname{Ad}^{*}_{\Gamma}(\mu, \gamma), 
(B, \nu)\right\rangle 
&= \left\langle(\mu, \gamma), \left(aBa^{-1} , 
\nu - \operatorname{Tr}\left( \sigma(a^{-1} )B \right) 
\right)\right\rangle\\
&= \operatorname{Tr}_{\rm{res}}\left( \left(a^{-1} \mu a - 
\gamma\sigma(a^{-1} )\right)B \right)  + \gamma \nu,
\end{align*} 
where $\Gamma=[g,h]_\alpha \in
\widetilde{\operatorname{GL}_{\rm res}^0}$. It 
follows that the coadjoint action of 
$\widetilde{\operatorname{GL}_{\rm res}^0}$ on the 
predual $\left(\mathfrak{gl}_{\rm res}\right) _*\oplus 
\mathbb{C}$ reads
\[
\operatorname{Ad}^{*}_{\Gamma}(\mu, \gamma) = 
\left(a^{-1} \mu a- \gamma\sigma(a^{-1} ), \gamma\right).
\]
This expression coincides with \cite[formula (2.50)]{GO10}. 

We now compute the coadjoint action of the full group 
$\widetilde{\operatorname{GL}_{\rm res}}$. It remains to 
compute the coadjoint action of the shift $r$ on 
$\left(\mathfrak{gl}_{\rm res}\right)_* \oplus \mathbb{C}$. 
From the general formula for the adjoint action of the 
semidirect product of two groups (see 
\cite[\S 6.4]{MaMiOrPeRa2007}), we get the adjoint 
action of $(1,  [\operatorname{id}_\mathcal{H},
\operatorname{id}_{\mathcal{H}_+}]_\alpha )\in \mathbb{Z} \,\ltimes\,
\widetilde{\operatorname{GL}_{\rm res}^0}$ on  
$\widetilde{\mathfrak{gl}_{\rm res}}$
\[
\operatorname{Ad}_{(1, [\operatorname{id}_\mathcal{H},
\operatorname{id}_{\mathcal{H}_+}]_\alpha )}
\left( [B, q]_\alpha  \right) = [ rBr^{-1}, \tilde{q}]_\alpha ,
\]
where $\tilde{q}:= r_{|\mathcal{H}_+} \circ q \circ 
\left(r_{|\mathcal{H}_+}\right)^{-1}$ on $r(\mathcal{H}_+)$ 
and $\tilde{q}:= 0$ on $\left(r_{|\mathcal{H}_+}(\mathcal{H}_+)
\right)^{\perp}=\mathbb{C}e_0$. By the identification 
$\widetilde{\mathfrak{gl}_{\rm res}} = \mathfrak{gl}_{\rm res}
\oplus \mathbb{C}$, one has for $(B, \nu)\in
\mathfrak{gl}_{\rm res}  \oplus \mathbb{C}$~:
\[
\operatorname{Ad}_{(1, [\operatorname{id}_\mathcal{H},
\operatorname{id}_{\mathcal{H}_+}]_\alpha )}(B, \nu) 
= \left( r Br^{-1}, \nu- 2i\langle e_{-1}, 
Be_{-1}\rangle \right) .
\]
By the duality pairing \eqref{pairing} between 
$\left(\mathfrak{gl}_{\rm res}\right)_*\oplus 
\mathbb{C}$ and $\mathfrak{gl}_{\rm res}\oplus \mathbb{C}$, 
one has
\begin{align*}
\langle\operatorname{Ad}^{*}_{(1, [\operatorname{id}_\mathcal{H},
\operatorname{id}_{\mathcal{H}_+}]_\alpha )}(\mu, \gamma), 
(B, \nu)\rangle &= 
\langle (\mu, \gamma), \operatorname{Ad}_{(1, [\operatorname{id}_\mathcal{H},
\operatorname{id}_{\mathcal{H}_+}]_\alpha )}(B, \nu)\rangle\\
&= \langle (\mu, \gamma), (r Br^{-1}, \nu - 2i\langle e_{-1}, 
Be_{-1}\rangle)\rangle \\
&=\operatorname{Tr}_{\rm{res}}\left( (r^{-1}\mu r) B \right)  
+ \gamma \nu -2i\gamma\langle e_{-1}, Be_{-1}\rangle.
\end{align*}
A direct computation yields 
\[
2i\langle e_{-1}, Be_{-1}\rangle = 
\operatorname{Tr}_{\rm{res}}\left( 
\left(r^{-1} d r-d\right)B\right)= 
\operatorname{Tr}_{\rm{res}} 
\left( \sigma ( r ^{-1} ) B \right),
\]
and hence
\[
\operatorname{Ad}^{*}_{(1, [\operatorname{id}_\mathcal{H},
\operatorname{id}_{\mathcal{H}_+}]_\alpha )}(\mu, \gamma) 
= \left(r^{-1}\mu r - \gamma\sigma(r ^{-1} ), \gamma\right).
\]
Formula \eqref{action coaff = action co} follows since 
it is verified for a set of generators of the group. 
Formula \eqref{ad_star} follows by taking its derivative.
\end{proof}

\begin{rem}{
The proposition implies that the coadjoint orbits of the central 
extension $\widetilde{\operatorname{GL}_{\rm{res}}}$ are affine 
coadjoint orbits of $\operatorname{GL}_{\rm res}$ on 
$\left(\widetilde{\mathfrak{gl}_{\rm{res}}}\right)_*$ for 
the following action:
 \begin{equation}\label{affine_coadjoint}
a\cdot\mu=a\mu a^{-1} -\gamma\sigma(a),
\end{equation}
where $\sigma : \operatorname{GL}_{\rm{res}}\to
\left(\mathfrak{gl}_{\rm{res}}\right)_*$ is the group 
1-cocycle given by
\[
\sigma(a)=ada^{-1}-d.
\]
More precisely, denoting by $\widetilde{\mathcal{O}}_{(\mu, \gamma)}$ 
the coadjoint orbit of the central extension 
$\widetilde{\operatorname{GL}_{\rm{res}}}$  through the 
element $(\mu, \gamma)\in
\left(\widetilde{\mathfrak{gl}_{\rm{res}}}\right)_*$ 
and by $\mathcal{O}^{d}_{\mu}$ the affine coadjoint orbit 
of the group $\operatorname{GL}_{\rm res}$ through the 
element $\mu\in \left(\mathfrak{gl}_{\rm{res}}\right)_*$, 
we have~:
\begin{equation}\label{tildeO=O}
\widetilde{\mathcal{O}}_{(\mu, \gamma)} 
= \mathcal{O}^{\gamma d}_{\mu} \oplus\{\gamma\}.
\end{equation}
}
\end{rem}

\subsection{The central extension of 
$\operatorname{Sp}_{\rm{res}}(\mathcal{V},\Omega)$} 
The central extension of $\operatorname{GL}_{\rm{res}}$ by 
$\mathbb{C}^\times$ restricts to a central extension 
$\widetilde{\operatorname{Sp}_{\rm{res}}}(\mathcal{V},\Omega)$ 
of the restricted linear symplectic group 
$\operatorname{Sp}_{\rm{res}}(\mathcal{V},\Omega)$ by 
$\mathbb{S}^1$. Indeed, for
\begin{equation}\label{notation}
A=
\begin{pmatrix}
A_1      &A_2\\
\bar{A}_2&\bar{A}_1
\end{pmatrix}
,\quad 
B=
\begin{pmatrix}
B_1      &B_2\\
\bar{B}_2&\bar{B}_1
\end{pmatrix}
\in\mathfrak{sp}_{\rm{res}}(\mathcal{V},\Omega),
\end{equation}
we have $s(A,B)=2i\operatorname{Tr}\left( \bar{A}_2B_2-
A_2\bar{B}_2\right) \in\mathbb{R}$. 

\section{The restricted Siegel disc as a coadjoint orbit}
\label{spres}

The Siegel disc $\mathfrak{D}(\mathcal{H})$ and 
the restricted Siegel disc $\mathfrak{D}_{\rm res}(\mathcal{H})$ have
been defined at the beginning of Section \ref{Siegel}. Recall from
Theorem \ref{disc restreint de Siegel = espace homogene} that
\begin{equation}
\label{formula_theorem_1_repeated}
\operatorname{U}(\mathcal{H}_+)\backslash
\operatorname{Sp}_{\rm res}(\mathcal{V},\Omega)
\longrightarrow \mathfrak{D}_{\rm res}(\mathcal{H}),\quad 
\left[
\begin{pmatrix}
g       &h\\
\bar{h} &\bar{g}\end{pmatrix}
\right]_{\operatorname{U}(\mathcal{H}_+)}\longmapsto h\bar{g}^{-1}
\end{equation}
is a diffeomorphism of real Banach manifolds.

Consider the coadjoint action of the central 
extension $\widetilde{\operatorname{Sp}_{\rm{res}}}(\mathcal{V},\Omega)$ 
on $\left(\widetilde{\mathfrak{sp}_{\rm{res}}}
(\mathcal{V},\Omega)\right)_* := 
\left(\mathfrak{sp}_{\rm{res}}(\mathcal{V},\Omega)\right)_*\oplus \mathbb{R}$ 
where
\[
\left(\mathfrak{sp}_{\rm{res}}(\mathcal{V},\Omega)\right)_*:=
\{\mu\in\mathfrak{sp}_{\rm{res}}(\mathcal{V}, \Omega)\mid 
p_\pm\mu|_{\mathcal{H}_\pm}\in L^1(\mathcal{H}_\pm)\},
\]
induced by the strongly non-degenerate pairing
\begin{equation*}
\begin{array}{lcll}
&\left(\left(\mathfrak{sp}_{\rm{res}}(\mathcal{V},\Omega)\right)_* 
\oplus \mathbb{R}\right)\times
\left(\mathfrak{sp}_{\rm{res}}(\mathcal{V},\Omega)\oplus \mathbb{R}\right)& 
\longrightarrow & \mathbb{R},\\
&\left((\mu,\gamma), (B, \nu)\right)& \longmapsto 
& \operatorname{Tr}_{\rm{res}}(\mu B) + \gamma \nu.
\end{array}
\end{equation*}
As in section~\ref{section_central_gl}, since the extension 
$\widetilde{\operatorname{Sp}_{\rm{res}}}(\mathcal{V},\Omega)$ 
of $\operatorname{Sp}_{\rm{res}}(\mathcal{V},\Omega)$ is central, 
the coadjoint action of 
$\widetilde{\operatorname{Sp}_{\rm{res}}}(\mathcal{V},\Omega)$ on 
$\left(\mathfrak{sp}_{\rm{res}}(\mathcal{V},\Omega)\right)_* \oplus \mathbb{R}$ 
descends to an affine coadjoint action of 
$\operatorname{Sp}_{\rm{res}}(\mathcal{V},\Omega)$  on 
$\left(\mathfrak{sp}_{\rm{res}}(\mathcal{V},\Omega)\right)_* \oplus \mathbb{R}$.

The isotropy group of $(0,\gamma)\in
\left(\widetilde{\mathfrak{sp}_{\rm{res}}}(\mathcal{V},\Omega)\right)_* 
=\left(\mathfrak{sp}_{\rm{res}}(\mathcal{V},\Omega)\right)_*
\oplus\mathbb{R}$, for $\gamma\neq 0$ is given by
\[
\operatorname{Sp}_{\rm{res}}(\mathcal{V},\Omega)_{(0,\gamma)}
=\left\{\left.
\begin{pmatrix}
g &0\\
0 &\bar{g}\end{pmatrix}
\right| g\in\operatorname{U}(\mathcal{H}_+)\right\}
\cong \operatorname{U}(\mathcal{H}_+),
\]
where $\operatorname{U}(\mathcal{H}_+)$ denotes the unitary 
group of $\mathcal{H}_+$. It is a closed Lie subgroup of 
$\operatorname{Sp}_{\rm{res}}(\mathcal{V},\Omega)$
and the quotient space $\operatorname{U}(\mathcal{H}_+)
\backslash\operatorname{Sp}_{\rm{res}}(\mathcal{V},\Omega)$ 
is naturally endowed with a Hilbert manifold structure.

Let $\mathcal{O}_{(0,\gamma)}$ be the coadjoint orbit of 
$(0,\gamma)\in\left(\widetilde{\mathfrak{sp}_{\rm{res}}}
(\mathcal{V},\Omega)\right)_*$, $\gamma\neq 0$, endowed with 
the Hilbert manifold structure induced by the diffeomorphism
\begin{equation}\label{diffeo2}
\begin{array}{lcll}
&\operatorname{U}(\mathcal{H}_+) \backslash 
\operatorname{Sp}_{\rm{res}}(\mathcal{V},\Omega)& 
\longrightarrow & \mathcal{O}_{(0,\gamma)},\\
&[a]_{\operatorname{U}(\mathcal{H}_+)}& \longmapsto 
& (-\gamma\sigma(a),\gamma)=
\left( -\gamma (ada^{-1}-d),\gamma\right).
\end{array}
\end{equation}
The tangent space at $(0,\gamma)\in\mathcal{O}_{(0,\gamma)}$ 
is given by
\[
T_{(0,\gamma)}\mathcal{O}_{(0, \gamma)}=\{-\gamma [A, d]\mid 
A\in\mathfrak{sp}_{\rm{res}}(\mathcal{V},\Omega)\}.
\]
The Kirillov-Kostant-Souriau symplectic form $\omega_\gamma$ 
on $\mathcal{O}_{(0,\gamma)}$ is given at the point $(0,\gamma)$ by
\begin{equation}
\label{KKS_value}
\omega_\gamma(0,\gamma)\left(-\gamma[A,d],-\gamma[B,d]\right)
=-\gamma\operatorname{Tr}(A[d,B]).
\end{equation} 
The tangent map of the diffeomorphism \eqref{diffeo2} at 
$[\operatorname{id}_{\mathcal{H}}]$ is given by
\[
\mathfrak{u}(\mathcal{H}_+)\backslash
\mathfrak{sp}_{\rm{res}}(\mathcal{V},\Omega)
\longrightarrow T_{(0,\gamma)}\mathcal{O}_{(0,\gamma)},
\quad  [A]_{\mathfrak{u}(\mathcal{H}_+)}  
\longmapsto -\gamma[A,d].
\]
More explicitly, using the notations \eqref{notation}, we have
\[\quad \left[
\begin{pmatrix}
A_1      &A_2\\
\bar{A}_2&\bar{A}_1\end{pmatrix} 
\right]_{\mathfrak{u}(\mathcal{H}_+)}\longmapsto 2i\gamma
\begin{pmatrix}
0        &A_2\\
-\bar{A}_2&0\end{pmatrix}.
\]
The value at 
$[\operatorname{id}_{\mathcal{H}}]_{\operatorname{U}(\mathcal{H}_+)}
\in  \operatorname{U}(\mathcal{H}_+)\backslash
\operatorname{Sp}_{\rm{res}}(\mathcal{V},\Omega)$ of the 
pull back $\widehat{\omega}_\gamma$ 
of the symplectic form $\omega_\gamma$
to $\operatorname{U}(\mathcal{H}_+) \backslash 
\operatorname{Sp}_{\rm{res}}(\mathcal{V},\Omega)$
by the diffeomorphism \eqref{diffeo2} has the expression
\begin{equation}\label{omega1}
\widehat{\omega}_\gamma\left( 
[\operatorname{id}_{\mathcal{H}}]_{\operatorname{U}(\mathcal{H}_+)}\right) 
\left( [A]_{ \mathfrak{u}( \mathcal{H} _+)},
[B]_{ \mathfrak{u}( \mathcal{H} _+)}\right) 
=-\gamma\operatorname{Tr}(A[d,B]),
\end{equation}
or, more explicitly,
\begin{multline}\label{omega_gamma}
\widehat{\omega}_\gamma\left( 
[\operatorname{id}_{\mathcal{H}}]_{\operatorname{U}(\mathcal{H}_+)}\right) 
\left(\left[
\begin{pmatrix}
A_1      &A_2\\
\bar{A}_2&\bar{A}_1\end{pmatrix} 
\right]_{ \mathfrak{u}( \mathcal{H} _+)},\left[
\begin{pmatrix}
B_1      &B_2\\
\bar{B}_2&\bar{B}_1\end{pmatrix} 
\right]_{ \mathfrak{u}( \mathcal{H} _+)}\right)
\\
=-2i\gamma\operatorname{Tr}(\bar{A}_2B_2-\bar{B}_2A_2).
\end{multline}
We summarize this discussion in the following Theorem.

\begin{thm}
The $\widetilde{\operatorname{Sp}_{\rm{res}}}(\mathcal{V},\Omega)$
coadjoint orbit 
\[
(\mathcal{O}_{(0,\gamma)}, \omega_\gamma) \quad\text{of}\quad 
(0,\gamma)\in\left(\widetilde{\mathfrak{sp}_{\rm{res}}}
(\mathcal{V},\Omega)\right)_*,\quad \gamma\neq 0
\]
 and the
restricted Siegel disc 
\[
\left(\mathfrak{D}_{\rm{res}}(\mathcal{H})
\cong \operatorname{U}(\mathcal{H}_+) \backslash 
\operatorname{Sp}_{\rm{res}}(\mathcal{V},\Omega), 
\widehat{\omega}_\gamma\right)
\]
 {\rm(}see 
\eqref{formula_theorem_1_repeated}{\rm)}  
are symplectically diffeomorphic
via the diffeomorphism \eqref{diffeo2}. The value
of $\omega_\gamma$ at $(0, \gamma)$ is given by
\eqref{KKS_value} and the value of
$\widehat{\omega}_\gamma$ at 
$[\operatorname{id}_{\mathcal{H}}]_{\operatorname{U}(\mathcal{H}_+)}
\in \operatorname{U}(\mathcal{H}_+)\backslash
\operatorname{Sp}_{\rm{res}}(\mathcal{V},\Omega)$ is given by
\eqref{omega1} or \eqref{omega_gamma}.
\end{thm}

\end{document}